\documentclass{amsart}
\usepackage{amssymb}
\newtheorem{theorem}{Theorem}[section]

\newtheorem{lemma}[theorem]{Lemma}

\newtheorem{conjecture}[theorem]{Conjecture}

\theoremstyle{definition}
\newtheorem{definition}[theorem]{Definition}
\newtheorem{example}[theorem]{Example}

\theoremstyle{remark}
\newtheorem{remark}[theorem]{Remark}

\numberwithin{equation}{section}

\begin{document}
\title[Subshifts with slowly growing numbers of follower sets]{Subshifts with slowly growing numbers of follower sets}

\begin{abstract}

For any subshift, define $F_X(n)$ to be the collection of distinct follower sets of words of length $n$ in $X$. Based on a similar result proved in \cite{OrmesPavlov}, we conjecture that if there exists an $n$ for which $|F_X(n)| \leq n$, then $X$ is sofic. In this paper, we prove several results related to this conjecture, including verifying it for $n \leq 3$, proving that the conjecture is true for a large class of coded subshifts, and showing that if there exists $n$ for which $|F_X(n)| \leq \log_2(n+1)$, then $X$ is sofic.

\end{abstract}

\date{}
\author{Thomas French}
\address{Thomas French\\
Department of Mathematics\\
University of Denver\\
2280 S. Vine St.\\
Denver, CO 80208}
\email{tfrench9@du.edu}
\author{Nic Ormes}
\address{Nic Ormes\\
Department of Mathematics\\
University of Denver\\
2280 S. Vine St.\\
Denver, CO 80208}
\email{normes@du.edu}
\urladdr{www.math.du.edu/$\sim$ormes/}
\author{Ronnie Pavlov}
\address{Ronnie Pavlov\\
Department of Mathematics\\
University of Denver\\
2280 S. Vine St.\\
Denver, CO 80208}
\email{rpavlov@du.edu}
\urladdr{www.math.du.edu/$\sim$rpavlov/}

\maketitle

\section{Introduction}
\label{intro}

Let $X$ be a subshift, i.e. a closed, shift-invariant subset of $A^{\mathbb{Z}}$ where $A$ is some finite set. 
In this paper, we consider \emph{follower sets} for words $w$ appearing in $X$. By the follower set
of $w$ we mean the set of all one-sided infinite sequences $s$ which may follow $w$ in some point of $X$; 
see Section~\ref{defns} for a formal definition. It is well known that the
number of distinct follower sets in $X$ is finite if and only if the subshift $X$ is 
sofic \cite{LindMarcus}. In this paper, we consider the question of whether 
a sufficiently slow growth rate in the number of distinct 
follower sets for words of length $n$ in $X$ implies that $X$ is a sofic subshift. 

More specifically, let $F_X(n)$ denote the set of distinct follower 
sets in $X$ for words of length $n$. We make the following conjecture:

\begin{conjecture}\label{mainconj}
For a subshift $X$, if there exists $n$ for which $|F_X(n)| \leq n$, then $X$ is sofic.
\end{conjecture}

We are unable to prove this conjecture presently, but prove some supporting results in this paper. 
Firstly, we prove that if $|F_X(n)| \leq \log_2(n+1)$ for some 
$n \geq 1$, then $X$ is sofic (Theorem \ref{log}). We prove a version of the conjecture relating to $|\bigcup_{\ell \leq n} F_X(n)|$ rather than $|F_X(n)|$ (Theorem ~\ref{morsehedlund}). We also prove Conjecture~\ref{mainconj}
for $n=1,2,3$ (Theorems~\ref{n=1}, \ref{n=2}, and \ref{n=3}). 

Conjecture~\ref{mainconj} is motivated by a number of results, among them the following classical theorem of Morse and Hedlund. 

\begin{theorem}[\cite{MorseHedlund}]\label{mhthm}
For a subshift $X$, if there exists an $n$ such that the number of
words of length $n$ is less than or equal to $n$, then $X$ is a finite collection of periodic points.
\end{theorem}

An equivalent formulation is that a bound of $n$ on the number of words of length $n$ 
implies a uniform bound on the number of words of length $n$. 
With this phrasing, Conjecture~\ref{mainconj} is equivalent to a version of Theorem~\ref{mhthm} 
with \lq \lq words of length $n$\rq \rq \ replaced by \lq \lq follower sets of words of length $n$\rq \rq.  

Another motivation is recent work of the three authors. 
In \cite{French}, building on an example of Delacourt, the first author showed that even when bounded, 
the sequence $\{|F_X(n)|\}$ can exhibit some surprising behavior. In particular, 
he showed that while the sequence $\{|F_X(n)|\}$ is always eventually periodic, it is not 
necessarily eventually constant. In fact, the gaps between consecutive terms in the periodic portion
can be arbitrarily prescribed.
Even more closely related, the second and third authors proved in \cite{OrmesPavlov} that Conjecture~\ref{mainconj}
holds if follower sets are replaced by so-called \emph{extender sets}. 
For a word $w$ appearing in $X$, the extender set of $w$ is the set of
all pairs $(p,s)$ of a left-infinite sequence $p$ and a right-infinite sequence $s$ such that 
the concatenation $pws$ forms a legal point in $X$. Let $E_X(n)$ denote the 
set of distinct extender sets for words of length $n$ in $X$.
\begin{theorem}[\cite{OrmesPavlov}]
Let $X$ be a subshift. The following are equivalent. 
\begin{enumerate}
\item $X$ is sofic
\item the sequence $\{|E_X(n)|\}$ is uniformly bounded
\item for some $n \geq 1$, $E_X(n) \leq n$. 
\end{enumerate}
\label{thm:OP}
\end{theorem}
It is then natural to ask whether the above holds for follower sets as well, though the question seems more difficult.

We remark that one obvious approach would be to attempt to use Theorem \ref{thm:OP} above. 
In other words, one might attempt to prove that a small number of follower sets implies a small
number of extender sets, and therefore soficity. Indeed, this still may be an avenue to a proof. 
However, in Example~\ref{dude}, we show that the sequence $\{|E_X(n)|\}$ may grow exponentially while
$\{|F_X(n)|\}$ grows linearly, meaning that this approach may not be enough on its own.

\section{Definitions and preliminaries}
\label{defns}
We begin with a list of definitions. 
Let $A$ denote a finite set, which we will refer to as our alphabet, elements of $A$ will be referred to as letters. 

\begin{definition}
A \textbf{subshift} $X$ on an alphabet $A$ is some subset of 
$A^{\mathbb{Z}}$ which is shift-invariant and closed in the product topology.
\end{definition}

\begin{definition}
A \textbf{word} over $A$ is a member of $A^n$ for some $n \in \mathbb{N}$, which we call the \textbf{length} of $w$. We use $\varnothing$ to denote the \textbf{empty word}, the word of length zero.
\end{definition}

\begin{definition} For any words $v \in A^n$ and $w \in A^m$, we define the \textbf{concatenation} $vw$ to be the word in $A^{n+m}$ whose first $n$ letters are the letters forming $v$ and whose next $m$ letters are the letters forming $w$.
\end{definition}

\begin{definition} For a word $u \in A^n$, if $u$ can be written as the concatenation of two words $u=vw$ then we say that 
$v$ is a \textbf{prefix} of $u$ and that $w$ is a \textbf{suffix} of $u$. 
\end{definition}

\begin{definition} 
The \textbf{language} of a subshift $X$, denoted by $L(X)$, is the set of all words which appear in points of $X$. For any finite $n \in \mathbb{N}$, define $L_n(X) = L(X) \cap A^n$, the set of words in the language of $X$ with length $n$.
\end{definition}

\begin{definition}
For any subshift $X$ on an alphabet $A$, and any word $w$ in the language of $X$, we define the \textbf{follower set of $w$ in $X$}, $F_X(w)$, to be the set of all right-infinite sequences $s \in A^\mathbb{N}$ such that the infinite word $ws$ occurs in some point of $X$. (Note that $F_X(\varnothing)$ is simply the set of all right-infinite sequences appearing in any point of $X$). Similarly, we define the \textbf{predecessor set of $w$ in $X$}, written $P_X(w)$, to be the set of all left-infinite sequences $p \in A^{-\mathbb{N}}$ such that $pw$ occurs in some point of $X$.


\end{definition}

\begin{definition}
For any word $w \in L(X)$, we say that $w$ is \textbf{shortenable} if there exists $v \in L(X)$ with strictly shorter length than $w$ such that $F_X(w) = F_X(v)$.
\end{definition}

\begin{definition}
For any subshift $X$ over the alphabet $A$, and any word $w$ in the language of $X$, we define the \textbf{extender set of $w$ in $X$}, $E_X(w)$, to be the set of all pairs $(p, s)$ where $p$ is a left-infinite sequence of symbols in $A$, $s$ is a right-infinite sequence of symbols in $A$, and $pws$ is a point of $X$. 


\end{definition}

\begin{definition}
For any positive integer $n$, define the set $F_X(n) = \{F_X(w) \ | \ w \in L_{n}(X)\}$. Thus the cardinality $|F_X(n)|$ is the number of distinct follower sets of words of length $n$ in $X$. Similarly, define $E_X(n) = \{E_X(w) \ | \ w \in L_n(X)\}$ and $P_X(n) = \{P_X(w) \ | \ w \in L_n(X)\}$, so that $|P_X(n)|$ and $|E_X(n)|$ are the numbers of distinct extender sets of words of length $n$ in $X$ and predecessor sets of words of length $n$ in $X$ respectively.
\end{definition}

\begin{definition}
A subshift $X$ is \textbf{sofic} if it is the image of a shift of finite type under a continuous shift-commuting map. 
\end{definition}

Equivalently, sofic shifts are those with only finitely many follower sets, that is, a shift $X$ is sofic iff $\{F_X(w) \ | \ w \text{ in the language of $X$}\}$ is finite (See Theorem 3.2.10 of \cite{LindMarcus}). The same equivalence exists for extender sets: X is sofic iff $\{E_X(w) \ | \ w \text{ in the language of $X$} \}$ is finite. (See Lemma 3.4 of \cite{OrmesPavlov})

\section{An example with many more predecessor sets than follower sets}
\label{example}

\begin{example}\label{dude}
There exists a subshift $X$ such that for every $n$, $|F_X(n)| = 2n + 1$ and for every $n > 6$, $|P_X(n)| \geq 2^{\lfloor n/4\rfloor }$. 
\end{example}

\begin{proof}

Define a labeled graph $G$ as follows: the vertex set is $V = \mathbb{Z}^+ = \{0,1,2,\ldots\}$. From any vertex $n$ are two outgoing edges: one leads to $n+1$ and is labeled with $U$ (for `up'), and the other leads to $\lfloor n/2 \rfloor$ and is labeled with $D$ (for `down'), unless it is the lone self-loop in the graph from $0$ to itself, in which case it is labeled with $E$ (for `equals.') Then, define a subshift $X$ with alphabet $\{D, U, E\}$ whose language consists of all labels of finite paths on $G$. For example, since from $17$ one could follow $D$ (to $8$), $D$ (to $4$), $U$ (to $5$), $U$ (to $6$), $D$ (to $3$), $D$ (to $1$), $D$ (to $0$), and then $E$ (to $0$), $DDUUDDDE$ would be a word in the language of $X$. An example of a word not in the language of $X$ would be $EUUDDD$, since $E$ must terminate at $0$, and then following $U$ would take you to $1$, another $U$ to $2$, $D$ to $1$, $D$ to $0$, and another $D$ is not legal from $0$. Note that $G$ is right-resolving, i.e. given an initial vertex $n$ and label $a$, there is at most one edge with initial vertex $n$ labeled by $a$.

We will need two auxiliary notations: for any $w \in L(X)$, denote by $T_G(w)$ the set of terminal states of paths in $G$ labeled by $w$, and by $I_G(w)$ the set of initial states of paths in $G$ labeled by $w$. Similarly, for any $n \in V$, denote by $F_G(n)$ the set of labels of right-infinite paths in $G$ with initial state $n$, and by $P_G(n)$ the set of labels of left-infinite paths in $G$ with terminal state $n$. It should be clear that for any $w$, $F_X(w) = \bigcup_{n \in T_G(w)} F_G(n)$, and $P_X(w) = \bigcup_{n \in I_G(w)} P_G(n)$.\\

\textbf{Claim 1:} For every $n$, $|F_X(n)| = 2n+1$.

Fix any $n$, and consider any $w \in L_n(X)$ which contains an $E$. Since $E$ can only terminate at $0$, and $G$ is right-resolving, there is only one possible terminal vertex of a path in $G$ labeled by $w$, and so $T_G(w) = \{k\}$ for some $k$, and correspondingly, 
$F_X(w) = F_G(k)$. Since the largest vertex in $V$ that can be reached from $0$ via a path of length less than $n$ is $n-1$, $0 \leq k \leq n-1$. We claim that for $k < k' \in [0,n-1]$, $F_G(k) \neq F_G(k')$; if $2^m$ is the smallest power of $2$ greater than $k$, then the reader may check that $U^{2^m - k -1} D^{m} E E E E \ldots $ is in $F_G(k)$ but not $F_G(k')$. Finally, we note that for each 
$k \in [0,n-1]$, the word $w = E^{n-k} U^k$ has $T_G(w) = \{k\}$, and so all $n$ of the distinct follower sets 
$F_G(0), F_G(1), \ldots , F_G(n-1)$ 
are in $F_X(n)$. 

Now, consider any $w \in L_n(X)$ which does not contain an $E$. We will prove by induction on $n$ that $T_G(w) = [k,\infty)$ for some $0 \leq k \leq n$. The hypothesis is easy for $n = 1$; $T_G(U)$ is clearly $[1,\infty)$, and $T_G(D)$ is similarly clearly $[0,\infty)$. Now, suppose that the inductive hypothesis is true for $n$, and consider $w \in L_{n+1}(X)$. We can of course represent $w = w' a$, where $w' \in L_n(X)$. By the inductive hypothesis, $T_G(w') = [k,\infty)$ for $0 \leq k \leq n$. The reader can verify that if $a = D$, then $T_G(w) = [\lfloor k/2 \rfloor, \infty)$, and if $a = U$, then 
$T_G(w) = [k+1, \infty)$, completing the inductive step. The proof is then completed, and so $F_G(w) = \bigcup_{i \geq k} F_G(i)$ for some $k \in [0,n]$. 

We claim that for 
$k < k' \in [0,n]$, $\bigcup_{i \geq k} F_G(i) \neq \bigcup_{i \geq k'} F_G(i)$; again, if $2^m$ is the smallest power of $2$ greater than $k$, then 
$U^{2^m - k - 1} D^{m} E E E \ldots$ is in $F_G(k)$ but not $F_G(i)$ for any $i \geq k'$. We note that for each $k \in [0,n]$, the word $w = D^{n-k} U^k$ has $T_G(w) = [k,\infty)$, and so follower sets 
$\bigcup_{i \geq k} F_G(i)$ are in $F_X(n)$ for $k = 0, 1, 2, \ldots, n$.

There are therefore $n$ distinct follower sets of words $w \in L_n(X)$ containing an $E$ and $n+1$ distinct follower sets of words $w \in L_n(X)$ not containing an $E$. To prove the claim that there are $2n+1$ in total, we verify that for all $k, k' \in [0,n]$, $\bigcup_{i \geq k'} F_G(i) \neq F_G(k)$. If $k \neq k'$, we have already distinguished these sets by identifying an element of 
$F_G(\min(k,k'))$ that is not in $F_G(i)$ for any $i > \min(k,k')$.
If $k' = k$ and $2^m>k$, then $D^{m+1} E E E \ldots$ is in $F_G(2^m) \subseteq \bigcup_{i \geq k} F_G(i)$, but not in $F_G(k)$.
This completes the proof of the claim. \\

\textbf{Claim 2:} For every $n > 6$, $|P_X(n)| \geq 2^{\lfloor n/4\rfloor }$. 

We will consider the set $S$ of all $w \in \{D,U,E\}^n$ which end with $D^{\lceil n/2 \rceil - 1} E$, contain no other $E$, and do not contain consecutive $U$ symbols. Clearly $|S|$ is the number of $\left(\lfloor n/2 \rfloor \right)$-letter words on $\{D,U\}$ without consecutive $U$ symbols, which is greater than or equal to $2^{\lfloor n/4 \rfloor}$ (simply freely choose the first letter, force the second to be $D$, freely choose the third letter, and so on). Therefore, it suffices to show that $S \subseteq L_n(X)$ and that for $w \neq w' \in S$, $P_X(w) \neq P_X(w')$. We will verify both claims by the auxiliary claim that for any $w \in S$, $I_G(w)$ is a nonempty finite interval of integers, and that if $w \neq w' \in S$, then $I_G(w) \neq I_G(w')$. Clearly the fact that $I_G(w) \neq \varnothing$ will imply that $w \in L_n(X)$. Also, if $I_G(w) \neq I_G(w')$, then one can choose $k \in I_G(w) \triangle I_G(w')$. The reader may check that $\ldots EEEU^k \in P_G(k)$ and is not in $P_G(k')$ for any $k' \neq k$, and so $P_X(w) = \bigcup_{i \in I_G(w)} P_G(i)$ and 
$P_X(w') = \bigcup_{i \in I_G(w')} P_G(i)$ are distinct. 

It remains only to prove the auxiliary claim. Consider any $w \in S$. We will work backwards from the end of $w$ to determine $I_G(w)$. First, write 
$w = v D^{\lceil n/2 \rceil - 1} E$. The reader may check that $I_G(D^{\lceil n/2 \rceil - 1} E) = [2^{\lceil n/2 \rceil - 2}, 2^{\lceil n/2 \rceil - 1})$, which we write as $[a,b)$ for brevity. Write $v = v_{\lfloor n/2 \rfloor} \ldots v_2 v_1$. Then, we will work from the right and state how each $v_i$ will alter this interval. For instance, if $v_1 = D$, then $I_G(v_1 D^{\lceil n/2 \rceil - 1} E)$ is the set of all vertices which lead to a vertex in $[a,b)$ via an edge labeled $D$, or $[2a, 2b)$. Similarly, if $v_1 = U$, then $I_G(v_1 D^{\lceil n/2 \rceil - 1} E)$ is the set of all vertices which lead to a vertex in $[a,b)$ via an edge labeled $U$, or $[a-1, b-1)$. In fact, it is simple to see in the same way that each $v_i$ will either double the endpoints of the interval (if $v_i = D$) or subtract one from the endpoints of the interval (if $v_i = U$). Since $a = 2^{\lceil n/2 \rceil - 2} > \lfloor n/2 \rfloor$ (for $n > 6$), clearly neither endpoint will ever go below $0$ in this procedure. 

This allows us to give a closed form for $I_G(w)$. Since $U$ does not change the length of the interval and $D$ doubles it, clearly the length of $I_G(w)$ is $2^j (b-a)$, where $j$ is the number of $D$ symbols in $v$. The left endpoint of $I_G(w)$, call it $c$, is obtained from $a$ via a sequence of either doubling or subtracting $1$, determined by whether the letters $v_1$, $v_2$, etc. are $D$ or $U$ respectively. Since no two consecutive $v_i$ can be $U$, two subtractions in a row are not permitted. For instance, if $a = 16$ and $v = UDDDUD$, then $c = (((((16 \cdot 2) - 1) \cdot 2) \cdot 2) \cdot 2) - 1 = 247$. It is not hard to check that this final answer could also be written as $c = 2^j a - (2^{n_1} + 2^{n_2} + \ldots + 2^{n_k})$, where $j$ is again the number of $D$ symbols in $v$, $k$ is the number of $U$ symbols in $v$, and $n_i$ is the number of $D$ symbols preceding the $i$th $U$ symbol in $v$. (Note that since $v$ does not contain consecutive $U$ symbols, $\{n_i\}$ is strictly decreasing). For instance, for $v = UDDDUD$, there are three $D$ symbols preceding the rightmost 
$U$ and zero $D$ symbols preceding the leftmost $U$, and so $n_1 = 3$ and $n_2 = 0$, yielding $c = 2^4 a - (2^{n_1} + 2^{n_2}) = 2^4 \cdot 16 - 8 - 1 = 247$. This yields the closed form
\[
I_G(w) = \left[2^j a - \sum_{i=1}^k 2^{n_i}, 2^j b - \sum_{i=1}^k 2^{n_i}\right).
\]
Then, if $w$ and $w'$ have different numbers of $D$ symbols, then the lengths of $I_G(w)$ and $I_G(w')$ are different, clearly implying that $I_G(w) \neq I_G(w')$. If $w$ and $w'$ have the same number of $D$ symbols, then the choices of $n_i$ for $w$ and $w'$ are distinct (since the $n_i$ uniquely determine $v$), and by uniqueness of binary representation, the sum $\sum_{i=1}^k 2^{n_i}$ would take different values for $w$ and $w'$, again implying that $I_G(w) \neq I_G(w')$. Therefore, all words in $S$ have distinct predecessor sets in $X$, and we are done.

\end{proof}

We note that since predecessor sets are just projections of extender sets, this example clearly has $|E_X(n)| \geq 2^{\lfloor n/4\rfloor }$ for every $n > 6$ as well, illustrating that in general, the number of extender sets of words of length $n$ may be much greater than the number of follower sets of words of length $n$.

\section{Main Results}
\label{proofs}

We begin with some simple facts about follower sets, which will repeatedly be useful in our analysis. The proofs are simple and left to the reader.

\begin{lemma}\label{unionslem}
For any subshift $X$, any $w \in L_n(X)$, and any $m \in \mathbb{N}$, $F_X(w) = \bigcup_{v} F_X(vw)$, where the union is taken over those $v \in L_m(X)$ for which $vw \in L_{m+n}(X)$. 
\end{lemma}

\begin{lemma}\label{superset}
For any subshift $X$, any $w \in L_n(X)$, and any $m<n$, 
there exists a $v \in L_m(X)$ for which $F(v) \supseteq F(w)$.
\end{lemma}

\begin{lemma}\label{lengthen}
Let $X$ be a subshift. If for two words $w, u \in L(X)$, $F_X(w) = F_X(u)$, then for any $v \in L(X)$, $F_X(wv) = F_X(uv)$.
\end{lemma}


The following will be our main tool for proving soficity of a subshift via the sets $F_X(n)$.

\begin{theorem}\label{unions}
For any subshift $X$, if there exists $n \in \mathbb{N}$ such that $\displaystyle F_X(n) \subseteq \bigcup_{\ell \leq n - 1} F_X(\ell)$, then $X$ is sofic. 
\end{theorem}

\begin{proof}
If there exists $n \in \mathbb{N}$ such that $\displaystyle F_X(n) \subseteq \bigcup_{\ell \leq n - 1} F_X(\ell)$, then for any word $w \in L_n(X)$, the follower set $F_X(w)$ is also the follower set of a strictly shorter word, so $w$ is shortenable to a word of length strictly less than $n$. Now, let $v \in L(X)$ of length greater than $n$, say $v = v_1v_2...v_nv_{n+1}...v_k$ where $k > n$. Then $v_1v_2...v_n \in L_n(X)$, and so is shortenable to some word $v' \in L(X)$ of length less than $n$. But $F_X(v_1v_2...v_n) = F_X(v')$ implies $F_X(v_1v_2...v_nv_{n+1}...v_k) = F_X(v'v_{n+1}...v_k)$ by Lemma~\ref{lengthen}, so $v$ is shortenable to a word $v'v_{n+1}...v_k$. If $v'v_{n+1}...v_k$ has length less than $n$, stop, but if $v'v_{n+1}...v_k$ has length at least $n$, we may apply the above process again and shorten repeatedly, getting shorter and shorter words with the same follower set until we find one with length less than $n$. So $v$ is shortenable to a word of length less than $n$. But this means that $\bigcup_{\ell \leq n-1} F_X(\ell)$ contains all follower sets in $X$, so $X$ has only finitely many follower sets, and thus, $X$ is sofic.
\end{proof}

We can now show that $|F_X(n)| = 1$ for any $n$ always implies soficity of $X$.

\begin{theorem}\label{n=1}
For any subshift $X$, if there exists $n$ for which $|F_X(n)| = 1$, then $X$ is a full shift.
\end{theorem}

\begin{proof}
We prove the contrapositive. Without loss of generality, assume that the alphabet $A$ of $X$ consists entirely of letters which actually appear in points of $X$, and assume that $X$ is not the full shift on $A$. Then there exists a word $w = w_1 w_2 \ldots w_k \in A^k$ which is not in the language of $X$; suppose that the length $k$ of $w$ is minimal. It must be the case that $k$ is at least $2$, since we assumed that all letters of $A$ are in $L(X)$. Then since we assumed $k$ to be minimal, $w_2 \ldots w_k \in L(X)$, so we can choose some one-sided infinite sequence $s$ appearing in $X$ which begins with $w_2 \ldots w_k$. Similarly, $w_1$ is in $L(X)$, so for any $n \in \mathbb{N}$, we may choose an $n$-letter word $v$ ending with $w_1$. Then $vs$ contains $w \notin L(X)$, so $s \notin F_X(v)$. However, since $s$ appears in $X$, there exists some $n$-letter word $u$ which can be followed by $s$ in $X$, and so $s \in F_X(u)$. Hence $F_X(u) \neq F_X(v)$, so $|F_X(n)| \geq 2$, and since $n$ was arbitrary, this is true for all $n$.
\end{proof}

We can now prove a version of Conjecture~\ref{mainconj} for unions of the sets $F_X(n)$, rather than the sets themselves.

\begin{theorem}\label{morsehedlund}
For any subshift $X$, if there exists $n \in \mathbb{N}$ so that $\displaystyle \left| \bigcup_{\ell \leq n} F_X(\ell) \right| \leq n$, then $X$ is sofic.
\end{theorem}

\begin{proof}
We prove the contrapositive, and so assume that $X$ is nonsofic. By Theorem~\ref{n=1}, $|F_X(1)| \geq 2$. Then, by Theorem~\ref{unions}, for every $n > 1$, there exists $S \in F_X(n) \setminus \bigcup_{\ell < n} F_X(\ell)$, and so $\left|\bigcup_{\ell \leq n} F_X(\ell) \right| > 
\left|\bigcup_{\ell \leq n-1} F_X(\ell) \right|$. Therefore, by induction, for each $n$, 
$\left|\bigcup_{\ell \leq n} F_X(\ell) \right| \geq n+1$. 
\end{proof}

We may now prove the following, which establishes a logarithmic lower bound for the growth rate of $|F_X(n)|$ for nonsofic shifts. 

\begin{theorem}\label{log}
For any subshift $X$, if there exists $n \in \mathbb{N}$ such that $|F_X(n)| \leq \log_2(n+1)$, then $X$ is sofic.
\end{theorem}

\begin{proof}
Suppose that for some $n$, $F_X(n) = \{F_1, F_2, ... , F_k\}$ where $k \leq \log_2(n+1)$. By Lemma~\ref{unionslem}, for each length $\ell < n$, every follower set of a word in $L_\ell (X)$ is a union of follower sets of words of length $n$. Therefore, every element of
$\bigcup_{\ell \leq n} F_X(\ell)$ is a non-empty union of elements of $F_X(n)$. There are at most $2^k - 1 \leq 2^{\log_2(n+1)} - 1 = n$ such unions, so $\left|\bigcup_{\ell \leq n} F_X(\ell) \right| \leq n$, which implies that $X$ is sofic by Theorem~\ref{morsehedlund}.
\end{proof}

Our next result shows that under the additional assumption that some non-empty word $w$ has the same follower set as the empty word, Conjecture~\ref{mainconj} is true.

\begin{lemma}\label{emptyword}
For any subshift $X$, if there exists a non-empty word $w \in L(X)$ such that $F_X(w) = F_X(\varnothing)$ and $n \in \mathbb{N}$ such that $|F_X(n)| \leq n$, then $X$ is sofic.
\end{lemma}

\begin{proof}
The follower set of the empty word is the set of all right-infinite sequences appearing in any point of $X$. If there exists a word $w$ such that any legal right-infinite sequence may appear after $w$, then by Lemma~\ref{superset}, there is a letter with this property as well. So we may assume that $F_X(a) = F_X(\varnothing)$ where $a$ is a single letter. 

The fact that $F_X(\varnothing) = F_X(a)$ implies by Lemma~\ref{lengthen} that for every $w \in L(X)$, 
$F_X(w) = F_X(aw) = F_X(aaw) = \ldots$. Therefore, every follower set of a word of length $\ell$ is also a follower set of a word of any length greater than $\ell$. In other words, $F_X(1) \subseteq F_X(2) \subseteq F_X(3) \subseteq \ldots$. 
Then, for every $n$, $F_X(n) = \bigcup_{\ell \leq n} F_X(\ell)$, and so if $|F_X(n)| \leq n$ for some $n$, clearly $\left| \bigcup_{\ell \leq n} F_X(\ell) \right| \leq n$, implying that $X$ is sofic by Theorem~\ref{morsehedlund}.
\end{proof} 

\begin{theorem}\label{n=2}
For any subshift $X$, if there exists $n \geq 2$ for which $|F_X(n)| \leq 2$, then $X$ is sofic.
\end{theorem}

\begin{proof}
The case where $|F_X(n)| = 1$ is treated by Theorem~\ref{n=1}, so we choose any $n \geq 2$ and suppose that there are exactly $2$ follower sets in $X$ of words of length $n$, say $F_1$ and $F_2$. We consider the sets in $F_X(1)$. By Lemma~\ref{unionslem}, every element of $F_X(1)$ is either $F_1$, $F_2$, or $F_1 \cup F_2$. If $|F_X(1)| = 1$, $X$ is sofic by Theorem~\ref{n=1}, so assume that $|F_X(1)| \geq 2$, that is, at least two of the above sets must appear in $F_X(1)$. Note that $F(\varnothing) = \bigcup_{w \in L_X(n)} F(w) = F_1 \cup F_2$, so by Lemma~\ref{emptyword}, if $F_1 \cup F_2$ is an element of $F_X(1)$, then $X$ is sofic. The only remaining case is that $F_X(1) = \{F_1, F_2\} = F_X(n)$, and then $X$ is sofic by Theorem~\ref{unions}.
\end{proof}

We are now prepared to prove Conjecture~\ref{mainconj} for $n \leq 3$.
Our proof is much more complicated than the cases where $n=1,2$. 

\begin{theorem}\label{n=3}
Let $X$ be a subshift. If $|F_X(n)| \leq n$ for any $n \leq 3$, then $X$ is sofic.
\end{theorem}

\begin{proof}
Clearly, for $n < 3$, Theorems~\ref{n=1} and \ref{n=2} imply this result. We can then restrict to the case where $n = 3$. If $|F_X(3)| < 3$, then $X$ is again sofic by either Theorem~\ref{n=1} or Theorem~\ref{n=2}. We therefore suppose that $|F_X(3)| = 3$, say $F_X(3) = \{F_1, F_2, F_3\}$. We also note that $F(\varnothing) = F_1 \cup F_2 \cup F_3$, and if any of $F_X(1)$, $F_X(2)$, or $F_X(3)$ contains $F_1 \cup F_2 \cup F_3$ as an element, then $X$ is sofic by Lemma~\ref{emptyword}. Therefore, in everything that follows, we assume that $F_1 \cup F_2 \cup F_3$ is not contained in 
$F_X(i)$ for $i \leq 3$.

We first show that if any $F_i$ is contained entirely within another, then $X$ is sofic. Suppose for a contradiction that some $F_i$ is contained in another, and so without loss of generality, we say that $F_2 \subseteq F_1$. By Lemma~\ref{unionslem}, all elements of $F_X(2)$ are nonempty unions of $F_1, F_2$, and $F_3$. However, $F_1 \cup F_3 = F_1 \cup F_2 \cup F_3$, and so $F_1 \cup F_3 = F_1 \cup F_2 \cup F_3$ are not in $F_X(2)$ as assumed above. Also, $F_1 \cup F_2 = F_1$. Therefore, the only possible elements of $F_X(2)$ are $F_1, F_2$, $F_3$, and $F_2\cup F_3$. If fewer than three of these four sets are part of $F_X(2)$, then $X$ is sofic by Theorem ~\ref{n=2}. Thus we may assume at least three of the four sets appear. If $F_1$, $F_2$, and $F_3$ are all in $F_X(2)$, then $F_X(3) \subseteq F_X(2)$, implying that $X$ is sofic by Theorem~\ref{unions}. Therefore, $F_2 \cup F_3 \in F_X(2)$. We note that by Lemma~\ref{superset}, some element of $F_X(2)$ must contain $F_1$. If $F_3$ contained $F_1$, then $F_3 = F_1 \cup F_2 \cup F_3$ is in $F_X(3)$, which we assumed not to be the case above. Similarly, $F_2 \cup F_3$ cannot contain $F_1$. Therefore, $F_1$ is the only set of $F_1$, $F_2$, $F_3$, and $F_2 \cup F_3$ to contain $F_1$, and so 
$F_1 \in F_X(2)$. Therefore $F_X(2)$ consists of $F_1$, $F_2\cup F_3$, and exactly one of $F_2$ and $F_3$. We note that if $F_2 \cup F_3$ is equal to any of $F_1$, $F_2$, or $F_3$, then either $|F_X(2)| = 2$ or $F_X(3) \subseteq F_X(2)$, in either case implying soficity by either Theorem~\ref{n=2} or Theorem~\ref{unions}. So from now on we assume $F_2\cup F_3$ is not equal to $F_1$, $F_2$, or $F_3$.

Now, let us consider $F_X(1)$. By Lemma~\ref{unionslem}, $F_X(1)$ can only consist of unions of sets in $F_X(2)$. 
The set $F_X(1)$ cannot contain $F_1 \cup F_3 = F_1 \cup F_2 \cup F_3$, and since $F_1 \cup F_2 = F_1$ 
we see that $F_X(1) \subseteq F_X(2)$. 
There exists some word $ab \in L_X(2)$ such that $F_X(ab) = F_2\cup F_3$. 
Clearly $F_X(a)$ is an element of $F_X(1)$ and therefore $F_X(a) = F_X(xy)$ for some $xy \in L_2(X)$. But then by Lemma~\ref{lengthen}, 
$F_X(xyb) =F_X(ab) = F_2 \cup F_3$, a contradiction since we above noted that $F_2 \cup F_3$ does not equal any of $F_1$, $F_2$, or $F_3$. 
We have then shown that if any of the follower sets $F_1, F_2$, and $F_3$ are contained in one another, $X$ is sofic, and so for the rest of the proof assume that no such containments exist. Note that this also implies that if any of $F_1 \cup F_2$, $F_1 \cup F_3$, or $F_2 \cup F_3$ contain each other, then the containing set is $F_1 \cup F_2 \cup F_3$, which we have assumed is not in $F_X(1)$, $F_X(2)$, or $F_X(3)$. 

We break the remainder of the proof into cases by how many of the sets $F_1, F_2$, and $F_3$ are elements of $F_X(2)$. If all three of the sets are elements of $F_X(2)$, then $X$ is sofic by Theorem~\ref{unions}. We then have three remaining cases.\\

\textit{Case 1: none of $F_1$, $F_2$, $F_3$ are in $F_X(2)$.} By Lemma~\ref{unionslem}, $F_X(2)$ consists of nonempty unions of 
$F_1, F_2$, and $F_3$, and we have assumed that $F_1 \cup F_2 \cup F_3$ is not in $F_X(2)$. If $|F_X(2)| \leq 2$, then $X$ is sofic by 
Theorem~\ref{n=2}. The only possibility is then that $F_X(2) = \{ F_1\cup F_2, F_1\cup F_3, F_2\cup F_3\}$. Then by Lemma~\ref{superset}, 
$F_X(1)$ must contain supersets of each of these sets, and it cannot contain $F_1 \cup F_2 \cup F_3$. This forces $F_X(1)$ to also be 
$\{ F_1\cup F_2, F_1\cup F_3, F_2\cup F_3\}$, meaning that $F_X(2) = F_X(1)$, and so $X$ is sofic by Theorem~\ref{unions}.\\

\textit{Case 2: exactly one of $F_1$, $F_2$, $F_3$ is in $F_X(2)$.} Without loss of generality, suppose that $F_1 \in F_X(2)$ and $F_2, F_3 \notin F_X(2)$. 
At least two other sets must be elements of $F_X(2)$ or else $X$ is sofic by Theorem~\ref{n=2}, and they must be unions of $F_1$, $F_2$, and $F_3$ by Lemma~\ref{unionslem}. Therefore, $F_X(2)$ contains at least two of the sets $F_1 \cup F_2$, $F_1 \cup F_3$, and $F_2 \cup F_3$. By Lemma~\ref{superset}, supersets of any such unions are also present in $F_X(1)$, which must be the sets themselves since we've assumed that $F_1 \cup F_2 \cup F_3 \notin F_X(1)$. If $F_1$ is also in $F_X(1)$, $F_X(2) \subseteq F_X(1)$, and $X$ would be sofic by Theorem~\ref{unions}, so $F_1 \notin F_X(1)$. 

Now, let $abc$ be some word such that $F_X(abc) = F_2$. What, then, is the follower set of $ab$? If it is any set in
$F_X(1)$, then there would exist $d$ so that $F_X(ab) = F_X(d)$, and then $F_X(abc)$ would equal $F_X(dc)$ by Lemma~\ref{lengthen}, meaning that $F_2 \in F_X(2)$, a contradiction. So the only choice for $F_X(ab)$ is $F_1$. Since at least two of $F_1 \cup F_2$, $F_1 \cup F_3$, and $F_2 \cup F_3$ are in $F_X(2)$, $F_X(2)$ contains a set of the form $F_1 \cup F_i$. Say that $F_X(xy) = F_1 \cup F_i$. Then, 
$F_X(xy) \supseteq F_X(ab)$, meaning that $F_X(xyc) \supseteq F_X(abc) = F_2$. Since none of the $F_i$ contain each other, this means that 
$F_X(xyc) = F_2$. But then since $F_1 \cup F_i$ also is a member of $F_X(1)$, there exists $z$ so that $F_X(z) = F_1 \cup F_i$, and then by Lemma~\ref{lengthen}, $F_X(zc) = F_2$, a contradiction since $F_2 \notin F_X(2)$. Hence, $X$ is sofic in this case as well.\\

\textit{Case 3: exactly two of $F_1$, $F_2$, $F_3$ are in $F_X(2)$.} Without loss of generality, suppose that $F_1, F_2 \in F_X(2)$ and 
$F_3 \notin F_X(2)$. 
By Lemma~\ref{superset}, $F_X(2)$ must contain some superset of $F_3$ which is not $F_1 \cup F_2 \cup F_3$, so it is of the form $F_3 \cup F_i$ for $i=1$ or $2$. As in Case $2$, any of the sets $F_1 \cup F_2$, $F_1 \cup F_3$, or $F_2 \cup F_3$ which is an element of $F_X(2)$ must be in $F_X(1)$ as well. This means that if $F_1$ and $F_2$ are both in $F_X(1)$, then $F_X(2) \subseteq F_X(1)$ and $X$ would be sofic by Theorem~\ref{unions}, so we restrict to the case where at least one of these sets is not in $F_X(1)$.

Now, let $abc$ be some word such that $F_X(abc) = F_3$. As in Case $2$, the follower set of $ab$ must be some set which occurs in $F_X(2)$ but not $F_X(1)$, which must be either $F_1$ or $F_2$ (depending on which is not part of $F_X(1)$). Without loss of generality, we say that 
$F_X(ab) = F_2$. We now show that neither $F_1 \cup F_2$ nor $F_2 \cup F_3$ is in $F_X(2)$. Suppose for a contradiction that there is a word $xy \in L(X)$ for which $F_X(xy) = F_2 \cup F_i$, $i = 1$ or $3$. Then, since $F(xy) \supseteq F(ab) = F_2$, $F(xyc) \supseteq F(abc) = F_3$. Again, since no $F_i$ contains another, this implies that $F(xyc) = F_3$. Finally, we note that $F(y) \supseteq F(xy) = F_2 \cup F_i$, so 
$F(y) = F_2 \cup F_i$. Therefore, by Lemma~\ref{lengthen}, $F(yc) = F(xyc) = F_3$, but this is a contradiction since $F_3 \notin F_X(2)$.
We now know that neither $F_1 \cup F_2$ nor $F_2 \cup F_3$ is in $F_X(2)$. By Lemma~\ref{unionslem}, all sets in $F_X(2)$ are nonempty unions of $F_1$, $F_2$, and $F_3$, and if $|F_X(2)| < 3$, then $X$ is sofic by Theorem~\ref{n=2}. The only remaining case is then that $F_X(2) = \{F_1, F_2, F_1\cup F_3\}$. 

We now consider the sets in $F_X(1)$. Recall that $F_2 \notin F_X(1)$ and that $F_1\cup F_3 \in F_X(1)$ since $F_1 \cup F_3 \in F_X(2)$. If 
$|F_X(1)| = 1$, then $X$ is sofic by Theorem~\ref{n=1}, so we can assume that $F_X(1)$ contains at least one other set, which must be a nonempty union of the elements of $F_X(2)$ by Lemma~\ref{unionslem}. The only possibilities are $F_1$ and $F_1\cup F_3$, since we assumed earlier that $F_1 \cup F_2 \cup F_3 \notin F_X(2)$. 
Therefore, every set in $F_X(1)$ is a superset of $F_1$. 

Our final step will involve considering what happens when a word with follower set $F_1$ is extended on the right by a letter. Suppose for a contradiction that there exists a word $w \in L(X)$ with $F_X(w) = F_1$ and a letter $i$ for which $F_X(wi) = F_2$. Then, for any letter $j$, since $F_X(j) \in F_X(1)$, $F_X(j) \supseteq F_X(w) = F_1$. Therefore, $F_X(ji) \supseteq F_X(wi) = F_2$. However, the only superset of $F_2$ in $F_X(2)$ is $F_2$ itself, and so for every $j \in A$, $F_X(ji) = F_2$. Finally, note that, by Lemma ~\ref{unionslem}, $F_X(i) = \bigcup_j F_X(ji) = F_2$, a contradiction since $F_2 \notin F_X(1)$.

Similarly, let's assume for a contradiction that there exists a word $w \in L(X)$ with $F_X(w) = F_1$ and a letter $i$ for which $F_X(wi) = F_3$. Then, choose a letter $j$ with $F_X(j) = F_1 \cup F_3$. Then, since $F_X(j) \supseteq F_X(w) = F_1$, $F_X(ji) \supseteq F_X(wi) = F_3$. However, the only superset of $F_3$ in $F_X(2)$ is $F_1 \cup F_3$, so $F_X(ji) = F_1 \cup F_3$. Then, since $F_X(j) = F_X(ji) = F_1 \cup F_3$, by Lemma~\ref{lengthen}, 
$F_X(jii) = F_X(ji) = F_1 \cup F_3$, a contradiction since $F_1 \cup F_3 \notin F_X(3)$.

This means that for every word $w \in L(X)$ with $F_X(w) = F_1$ and any letter $a$ for which $wa \in L(X)$, $F(wa) = F_1$. But then, since the follower set of every letter contains $F_1$, the follower set of every legal $2$-letter word contains $F_1$, a contradiction since $F_X(2)$ contains $F_2$, and we assumed that none of the $F_i$ contains another. Every case has either led to a contradiction or to the conclusion that $X$ is sofic, and so we've proved that $X$ is sofic.

\end{proof}

Our final result is a version of Conjecture~\ref{mainconj} for a class of coded subshifts.
Recall the definition of coded subshifts below.

\begin{definition}
Given a set $\mathcal{W}$ of finite words, the \textbf{coded subshift} with \textbf{code words} $\mathcal{W}$ is the subshift generated by taking the closure of the set of all biinfinite sequences made from concatenating words in $\mathcal{W}$.
\end{definition}

\begin{theorem}\label{coded}
Given a sofic shift $X$, choose a subset $\mathcal{W} \subseteq L(X)$ with the property that for any finite word $v \in L(X)$, there exists some $w \in \mathcal{W}$ such that $v$ is a suffix of $w$. Create a coded subshift $Y$ with code words $\{wc \> | \> w \in \mathcal{W}\}$ where $c$ is a letter not appearing in the alphabet of $X$. Then if $|F_Y(n)| \leq n$ for any $n \in \mathbb{N}$, then $Y$ is sofic.
\end{theorem}

\begin{proof}

We begin with two preliminary observations. Firstly, $X \subseteq Y$, since any point of $X$ is a limit of finite words in $L(X)$, all of which are suffixes of code words, which are themselves in $L(Y)$. We also note that any word in $L(Y)$ without a $c$ must be a subword of a code word, and therefore in $L(X)$. 

Secondly, for any word $ucv \in L(Y)$, $F_Y(ucv) = F_Y(cv)$. Clearly $F_Y(ucv) \subseteq F_Y(cv)$. Let $s \in F_Y(cv)$. Because
$uc$ is the suffix of a concatentation of code words and $vs$ is the beginning of a concatenation of code words in $Y$, $ucvs$ occurs in $Y$ and therefore $F_Y(ucv) \supseteq F_Y(cv)$.

We begin our proof by claiming that there are only finitely many follower sets in $Y$ of words not containing the letter $c$. 
Given any word $w \in L(Y)$, if $w$ does not contain a $c$ then $w \in L(X)$. 
There are only finitely many follower sets in $X$, so it is sufficient to show that for any $w, v \in L(X)$, $F_X(w) = F_X(v)$ implies $F_Y(w) = F_Y(v)$. To that end, let $F_X(w) = F_X(v)$ and consider any $s \in F_Y(w)$. If $s$ does not contain the letter $c$, then $ws$ is a limit of longer and longer words in $\mathcal{W}$, and since all such words are in $L(X)$, $ws$ occurs in $X$, i.e. $s \in F_X(w)$. Since $F_X(w) = F_X(v)$, $s \in F_X(v)$, i.e. $vs$ also occurs in $X$. Since $Y \supseteq X$, $vs$ occurs in $Y$ as well, and so $s \in F_Y(v)$. 

On the other hand, if $s$ contains the letter $c$ and $s \in F_Y(w)$, then $s = s'cs''$ for some $s'$ not containing $c$ ($s'$ may be the empty word). By the same logic as above, $ws' \in L(X)$, therefore $vs' \in L(X)$, and so $vs'$ occurs as a suffix of some word in $\mathcal{W}$. But then, $vs' c$ is a suffix of some code word, and so $vs'cs''$ occurs in $Y$. 

We have shown that in both cases, $s \in F_Y(w)$ implies $s \in F_Y(v)$, and so $F_Y(w) \subseteq F_Y(v)$. By the same argument, $F_Y(v) \subseteq F_Y(w)$, giving $F_Y(w) = F_Y(v)$. Therefore there are only finitely many follower sets in $Y$ of words not containing $c$. 

Now, we assume that $n$ is such that $|F_Y(n)| \leq n$. Partition $L_n(Y)$ into $n+1$ sets based on the last appearance of the letter $c$ in the word--the first set $S_0$ consists of words with no $c$, the second set $S_1$ consists of words ending with $c$, the third $S_2$ consists of words ending with $c$ followed by another letter that is not $c$, and so on, up to the final set $S_n$ which consists of words beginning with a $c$ followed by $n-1$ other symbols which are not $c$. Since $X \subseteq Y$, there exist words in $L(Y)$ of every length without any $c$ symbols, implying that $S_0 \neq \varnothing$. Therefore, there must exist $k > 0$ so that all follower sets 
(in $Y$) of words in $S_k$ are also follower sets (in $Y$) of some word in $S_i$ for some $i < k$; else each of the $n+1$ sets $S_i$ would contribute a follower set not in any previous one, contradicting $|F_Y(n)| \leq n$. 

Let $w$ be a word in $L(Y)$ of length at least $k$. Our goal is 
to show that $F_Y(w)$ is either equal to one of the finitely many follower sets of words without a $c$ or to 
the follower set of a word of length less than $k$.
Clearly, if $w$ does not contain a $c$, we are done, so suppose $w$ contains the letter $c$. As noted earlier, 
$F_Y(w)$ is unchanged if all letters before the last occurrence of $c$ are removed from $w$. 
If this removal results in a word of length less than $k$, then again we are done. 
So let us proceed under the assumption that $w$ begins with $c$, has length $k$ or greater and contains no other $c$ symbols.

Let $p$ denote the $k$-letter prefix of $w$. 
Since $p$ begins with $c$, $p$ can be arbitrarily extended backwards in any way to yield an $n$-letter word $p'$ which has the same follower set as $p$. Note that $p' \in S_k$, and so there exists $i < k$ and $p'' \in S_i$ so that $F_Y(p) = F_Y(p') = F_Y(p'')$. There are two cases. If $i \neq 0$, then we may again remove the letters of $p''$ before the final $c$ symbol to yield a word $p'''$ of length $i < k$ for which $F_Y(p) = F_Y(p''')$. Then, we replace the prefix $p$ of $w$ by $p'''$ to yield a new word $w'$ with strictly smaller length, which still begins with a $c$ and contains no other $c$ symbols, and for which 
$F_Y(w) = F_Y(w')$ by Lemma ~\ref{lengthen}. We then repeat the above steps. If at each step, $i \neq 0$, then eventually $w$ will be shortened to a word of length at most $k$ with the same follower set in $Y$, of which there are clearly only finitely many.

The only other case is that at some point, the prefix of length $k$ has the same follower set in $Y$ as a word in $S_0$. Then, again by Lemma ~\ref{lengthen} that prefix can be replaced by the word in $S_0$, yielding a word with no $c$ symbols with the same follower set 
in $Y$ as $w$. There are only finitely many follower sets in $Y$ of words not containing $c$. We have then shown that $F_Y(w)$ (for arbitrary $w$ of length at least $k$) has follower set in $Y$ from a finite collection (namely all follower sets in $Y$ of words with no $c$ and all follower sets in $Y$ of words with length at most $k-1$), which implies that $Y$ is sofic.

\end{proof}

\begin{remark}
The class of coded subshifts $Y$ which may be created as in Theorem~\ref{coded} includes all so-called $S$-gap shifts ~\cite{LindMarcus} (with $X = \{0^{\infty}\}$ and $c = 1$) and the reverse context-free shift of \cite{Pavlov} (with $X = \{a,b\}^{\mathbb{Z}}$ and $c$ as in the theorem).

\end{remark}

\bibliographystyle{amsplain}
\bibliography{followers}

\end{document}